\documentclass [12pt,a4paper,reqno]{amsart}
\newcommand\Cref[1]{{Corollary~\ref{#1}}}
\newcommand{\etype}[1]{\renewcommand{\labelenumi}{(#1{enumi})}}
\def\eroman{\etype{\roman}}
\newtheorem{thm}{Theorem}[section] %the resolution could also be [subsection]
\newtheorem{conj}[thm]{Conjecture}
\newtheorem{cor}[thm]{Corollary}
\newtheorem{corollary}[thm]{Corollary}

\newtheorem{defn}[thm]{Definition}

\newtheorem{example}[thm]{Example}

\newtheorem{lem}[thm]{Lemma}
\newtheorem{prop}[thm]{Proposition}
\newtheorem{proposition}[thm]{Proposition}

\newtheorem{rem}[thm]{Remark}
\newtheorem{rems}[thm]{Remarks}

\def\GKdim{\operatorname{GK dim}}

\def\Ann{\operatorname {Ann}}

\def\({\left(}
\def\){\right)}

\def\SL{\operatorname{SL}}
\def\gl{\operatorname{gl}}
\def\sli{\operatorname{sli}}

\long\def\forget#1\forgotten{}
\newcommand\comp[3][\bullet]{{{#1}_{{\if1#2{}\else{#2}\fi}{\if#3K{}\else{(#3)}\fi}}}} % \comp[I]{2}{3} means component I is M_2(\F_{q^3}). Shorthands: suppress rank=1 and dimension=K.

\usepackage[usenames]{color}

\begin{document}
\title[Hopfian and Bassian algebras] {Hopfian and Bassian algebras}

\author{Louis Rowen}

\address{Department of Mathematics, Bar-Ilan University, Ramat-Gan
52900,Israel} %
\email{rowen@math.biu.ac.il}

\author{Lance Small}
\address{Department of Mathematics, University of California at
San Diego,
La Jolla} %
\email{lwsmall@math.ucsd.edu}
  \thanks{This work was supported by the U.S.-Israel Binational Science
Foundation (grant no. 2010149)}.
\thanks{The authors would like to thank Donald Passman and Eliezer Hakatan for helpful
conversations involving Theorem~\ref{Rep11}.}

%******************************* date *************************************
\date{\today}

%******************************* keywords *********************************


\begin{abstract} A ring $A$ is \textbf{Hopfian}  if  $A$  cannot be
isomorphic to a proper homomorphic image $A/J$. $A$ is
\textbf{Bassian}, if there cannot be an injection of $A$ into a
proper homomorphic image $A/J$. We consider classes of Hopfian
 and  Bassian rings, and tie representability of algebras and chain conditions on ideals to these
properties. In particular, any semiprime  algebra satisfying the ACC on semiprime
 ideals is
Hopfian, and any semiprime  affine
 PI-algebra  over a field   is Bassian. We also characterize the ACC on semiprime
 ideals.
\end{abstract}

\maketitle

\section{Introduction}
By ``affine'' we mean finitely generated as an algebra over a
Noetherian commutative base ring, often a field. In this paper we are
interested in structural properties of affine algebras, especially
those satisfying a polynomial identity (PI). Although there is an
extensive literature on affine PI-algebras, several properties
involving endomorphisms have not yet been fully investigated,
so that is the subject of this note. %We start by
%calling a faithful representation $\rho :A \to M_n(C)$
%\textbf{GK-compatible} if $\GKdim(A) = \GKdim(C).$ In such a case,
%one can study representability in terms of the commutative algebra
%over which it is represented.
%
%\medskip
%
% \noindent {\bf Proposition~\ref{GKcomp}.}  Any Noetherian algebra $A$ finite over its center is GK-compatible.
%
%\bigskip

 The main objective in this paper is to study the Hopfian
property, that an algebra~$A$ cannot be isomorphic to  a proper
homomorphic image $A/J$ for any $0 \ne J \triangleleft A.$ Our
specific motivation is to gain information about group algebras and
enveloping algebras. We collect, unify, and extend some known facts
and verify the Hopfian property in one more case:
 \bigskip

 \noindent {\bf Theorem~\ref{Hopf}.} Every   weakly representable affine
 algebra  over a commutative Noetherian ring  is Hopfian.

It is easy to see that any semiprime Hopfian ring satisfies  the ACC
on semiprime ideals, leading us to characterize this condition. (One
direction is standard, but we do not have a reference for the other
direction.)

  \bigskip
   \noindent {\bf Proposition~\ref{ACCsemip}.} A ring $A$ satisfies  the ACC on semiprime ideals, iff $A$  satisfies  the ACC on prime ideals, with each   semiprime ideal being the intersection of
 finitely many prime ideals.

  \bigskip
 We then turn to an even stronger property, the \textbf{Bassian
property}, that there cannot be an injection of $A$ into a proper
homomorphic image of $A$.

 \medskip

  \noindent {\bf Theorem~\ref{Rep11}.}
 Any semiprime  affine
 PI-algebra $A$ over a field $F$ is Bassian.

  \bigskip

%
%\section{A review of GK dimension and  growth}
%
%
%Our generic notation for this section: $A = F\{ a_1, \dots, a_\ell
%\}$ is an affine algebra over $F$, with lower nilradical $N$. $\{P_i
%: i\in I\}$ denotes the set of minimal nonzero prime ideals of $A$.
%Then $A$ is a subdirect product of the $A/P_i.$
%
%
%\subsection{The standard filtration of an affine algebra}
%
%Note that if $A$ is  PI or left Noetherian then $I$ is finitely
%generated as an ideal.

To compare algebras and their isomomorphisms, we need a nicely
behaved dimension with which to work, and for this purpose recall
the \textbf{Gelfand-Kirillov dimension} of an affine algebra $A =
F\{ a_1, \dots, a_\ell \}$, which is $ \GKdim (A): = \underset{n \to
\infty}\varlimsup\log _n d_n, $ where $A_n = \sum _n F a_{i_1}\cdots
a_{i_n}$ and $d_n = \dim _F A_n$.

The standard reference on Gelfand-Kirillov dimension is \cite{KrL},
from which we recall some  well-known facts:

\begin{itemize}
\item $\GKdim (A)$ is independent of the choice of generating set
$\{ a_1, \dots, a_\ell \}$ of $A$.

\item If $A \subset A'$ then $\GKdim (A) \le \GKdim (A').$

\item $\GKdim (A)  \ge \GKdim (A/\mathcal I)$ for every ideal
$\mathcal I$ of $A$.

\item  $\GKdim (A) > \GKdim (A/\mathcal I)$ for every ideal
$\mathcal I$ of $A$ containing a regular element of $A$.

\item  If $A$ is a finite subdirect product of algebras
$A_{\mathcal I_1}, \dots, A_{\mathcal I_t},$ then $$\GKdim (A) =
\max \{ \GKdim (A/\mathcal I_j): 1 \le j \le t \}.$$

%
%\item $A/N$ is a subdirect product of the $A/P_i$. In particular,
%when $I$ is finite, then $\GKdim (A) = \max \{ \GKdim (A/P_j): 1
%\le j \le t \}.$

%\item $\GKdim (A)=\GKdim (\End _A M)$, for any finite $A$-module
%$M$.

\item  The  GK dimension   of any algebra finite over $A$ equals
$\GKdim (A).$
\end{itemize}

\section{Hopfian rings}
\begin{defn}% A ring $A$ is \textbf{Hopfian} if every epimorphism
% $A \to A$ is an isomorphism.

  A ring is
\textbf{residually finite} if it is a subdirect product of finite
rings.

  An algebra over a field is
\textbf{residually finite dimensional} if it is a subdirect product of finite
dimensional algebras.
 \end{defn}

\begin{rems}\label{Hopf rem}$ $
 \begin{enumerate} \eroman\item Any finite ring is Hopfian, since any onto
map must be 1:1 by the pigeonhole principle. \item \cite{O} Any
commutative affine algebra over a field is Hopfian.
\item \cite{OR} Any
commutative affine algebra over a commutative ring $C$ is Hopfian
with respect to $C$-homomorphisms.
\item \cite{AFS} There is an affine PI-algebra over a field, that is not Hopfian.
\item  \cite{Ne}  There is a group which is not
Hopfian in the group-theoretic sense, so its group algebra is not
Hopfian.
\item  \cite{Ab} Furthermore, there is a  finitely presented solvable group which is not Hopfian  in the
group-theoretic sense.
 \item
Malt'sev proved that the free algebra is Hopfian, a fact used in
\cite{DL} in conjunction with the Jacobian conjecture. Short proofs
are given in \cite{C}, \cite{OR}, and \cite{BelRY}, the latter using
GK dimension.
\item \cite[Theorem~3]{Le} Every residually finite ring is Hopfian,
proved by Lewin via the fact that any subring of finite index in a
finitely generated ring is also finitely generated.
\item  Malt'sev proved that  residually finite dimensional    (called {\it
 approximable} in~\cite{BLK,BLK2}) affine algebras over a field are all Hopfian.
 \item  Positively graded affine algebras are residually finite
 dimensional, and thus Hopfian by (x).
\item Malt'sev also proved that every  representable affine algebra
over a field is residually finite dimensional and thus Hopfian,
cf.~\cite[Theorems 5 and~6]{BLK}.
\item The enveloping algebra of the positive part of the Witt
algebra is  positively graded and thus Hopfian. But it is not known
whether the enveloping algebra of the Witt algebra itself is Hopfian.

\item On the other hand, Irving \cite{I} produced a graded affine PI-algebra over a field,
 that is not representable but is positively graded and thus Hopfian.
 \item \cite{Mi1,Mi2} If the enveloping algebra of a Lie algebra $L$ is residually finite dimensional,
 then $L$ is residually finite dimensional as a Lie algebra.
\end{enumerate}
\end{rems}

For a given property designated P, we say  that an algebra $A$
satisfies ACC(P) if  any   ascending chain  of ideals $J_1 \subseteq
J_2 \subseteq \dots $ for which each $A/J_i$ has property P, must
stabilize. The following obvious observation is crucial.

\begin{lem}\label{makeswork} For any given property P, if every proper homomorphic image of $A$ satisfies ACC(P), then $A$
is Hopfian.\end{lem}
\begin{proof} Otherwise take $J_1$ such that $\varphi: A \cong A/J_1$, and
$J_2 \supset J_1$ such that $A \cong A/J_1 \cong  (A/J_1)/ (J_2/J_1)
\cong A/J_2$ and so forth, but $J_1 \subset J_2 \subset \dots,$ a
contradiction.
\end{proof}

 In
particular, we have the following instant consequences:

\begin{cor} Any algebra satisfying the ACC on two-sided ideals is
Hopfian.  \end{cor}

\begin{cor}\label{GKcomp11} Any prime ring $A$ satisfying ACC(prime ideals) is Hopfian.\end{cor}

We say that an ideal $P$ is  \textbf{completely prime} if $A/P$ is a
domain.

\begin{cor}\label{GKcomp12} Any domain $A$ satisfying ACC(completely prime ideals)  is Hopfian.\end{cor}
\begin{proof} Otherwise taking $\varphi: A \cong A/J \subseteq A, $ we see that $J$ is a completely prime ideal, so we use the
previous argument of Lemma~\ref{makeswork}.\end{proof}

The reason we became interested in Hopfian rings is in the converse,
that any non-Hopfian domain cannot satisfy ACC(completely prime
ideals), so if certain enveloping algebras are non-Hopfian, they
cannot satisfy ACC (ideals), by~Remark~\ref{Hopf rem}. But the
opposite is the case.

\begin{rem}\label{ansIS}  Iyudu and Sierra \cite[Theorem 1.2,
Proposition 6.4]{IS} have verified the ACC on completely prime
ideals for the most important of these enveloping algebras the
positive Witt algebra $U(W+),$the
  Witt algebra  $U(W),$ and the Virasoro algebra $U(Vir)$. Their method is to show that
  the first two
  algebras and central factors of $U(Vir)$
have just infinite GK-dimension; this instantly implies the
ACC(completely prime ideals) by an easy argument (since modding out
by the first completely prime ideal in an ascending chain would
yield a domain of finite GK dimension). Their argument for $U(Vir)$
is somewhat more intricate. One concludes that these algebras are
all Hopfian, by Corollary~\ref{GKcomp12}.
 \end{rem}

\begin{rem} Along the same lines, Passman and Small \cite[Lemma 1.1]{PS} proved
that for any epimorphism of finitely presented algebras, the kernel
is a finitely generated two-sided ideal. Hence if a finitely
presented algebra  satisfies the ACC on finitely generated two-sided
ideals, it must be Hopfian.\end{rem}

We turn to semiprime rings.
Suppose $J\triangleleft A.$ $\sqrt {J}$, called  the \textbf{prime radical} of the ideal $J$, denotes the intersection of the prime
 ideals containing ${J}$. The ideal ${J}$ is \textbf{semiprime} if $\sqrt {J}= J.$ A ring $A$ is \textbf{semiprime} if $\sqrt {0}= 0$.

\begin{cor}\label{ACCpr} Any semiprime ring $A$ satisfying the ACC on semiprime
 ideals is
Hopfian.   \end{cor}   \begin{proof} Non-Hopfian now says $J$ is a
nonzero ideal such that $A \cong A/J$. Then clearly  the ideal~$J$
is semiprime, so Lemma~\ref{makeswork} is applicable.
\end{proof}

This applies in particular to affine PI-algebras (although this
result is strengthened in Theorem~\ref{Rep11}), and many enveloping
algebras of Lie algebras. Note that Remark~\ref{Hopf rem} gives a
counterexample for non-semiprime rings.

Corollary~\ref{ACCpr} motivates us to study the  ACC on semiprime
 ideals. There is a natural criterion in terms of the prime spectrum.

%
% We do this by means of (left) annihilator ideals, since we know that for $A$ semiprime,
% the left and right annihilators of an ideal are the same, and
% $\Ann (\Ann P) = P$ for any  prime ideal $P$ with a nonzero annihilator. (Indeed,  $\Ann (\Ann P) \supseteq P.$ But
% $\Ann (\Ann P) \Ann P =0\subseteq P,$ so we are done unless $\Ann P \subseteq P, $
% implying $\Ann P^2 = 0,$ so $\Ann P = 0$.)

We say that a decomposition $J = P_1 \cap \dots P_t$ as an
intersection of prime ideals is \textbf{irredundant} if $P_i \not
\subseteq P_j$ for each $i,j$. The following observation is from
Herstein \cite{H}.

  \begin{lem}\label{irredu} Any irredundant decomposition $J = P_1 \cap \dots P_t$ of $J$ is unique (up to permutation), and thus consists of minimal prime ideals. \end{lem}

  \begin{prop}\label{ACCsemip} A ring $A$ satisfies  the ACC on semiprime ideals, iff $A$  satisfies  the ACC on prime ideals, with each   semiprime ideal being the intersection of
 finitely many prime ideals.  \end{prop}
\begin{proof} In either direction, we may assume that $A$ is semiprime by modding out some ideal  in the chain.
$(\Rightarrow)$ is a standard Noetherian-type induction argument, cf.~\cite[Proposition~2.3]{Pr}.%
%Otherwise, replacing $A$ by $A/J$, where $J\triangleleft A$ is a semiprime ideal maximal with respect to
%the assertion failing on $A/J$, we may assume that   each nonzero semiprime ideal $B$ of $A$ is a finite intersection of prime ideals.
%We are done if $0$ is a prime ideal, so we may assume that $B_1 B_2= 0$ for nonzero ideals  $B_1 B_2= 0$. By assumption, $\sqrt {B_i}$ is a finite
%intersection of prime ideals of $A$, so $0 = \sqrt{0} = \sqrt{B_1}\sqrt{B_2}$ is a finite
%intersection of prime ideals of $A$.

$(\Leftarrow)$ We define the \textbf{depth} of a semiprime ideal $J$ to be the  size of an irredundant decomposition, i.e., the minimal number of prime ideals whose intersection is $J$, and
for any infinite
 ascending chain  of semiprime ideals $0 = J_1 \subseteq
J_2 \subseteq \dots $, we induct on the  depth vector $\bold d =(d_1, d_2,\dots)$ (lexicographic order) where  $d_i$ is the depth of $J_i$.
Namely, we assume that there are no infinite chains of depth $< \bold d$, and want to prove that there are no infinite chains of depth $\bold d$.
(The assumption is vacuous for $d_1=1$)

We take a semiprime ideal $P$ of depth $\bold d$ maximal with respect to the assertion failing in $A/P$; passing to $A/P$ we may assume that the
assertion holds in every proper homomorphic image of $A.$ Write $J_1 = P_1 \cap \dots \cap P_t$ and $J_2 = Q_1 \cap \dots \cap Q_{t'}$
for prime ideals $P_j$ and $Q_j$
As in the proof of Lemma~\ref{irredu}, for each $Q_j$ we have some $P_{i(j)}\subseteq Q_j.$
Clearly $\cap _j P_{i(j)} \subseteq  \cap _j Q_j = J_2,$ so we could discard any $i$ not appearing as an $i_j$ and conclude by induction on depth.
Hence all the $i$ appear as $i(j)$ for suitable $j$, and we could build an infinite ascending sequence of prime ideals, contrary to assumption.
\end{proof}

%$\Ann P^2 = 0
%\begin{lem} Suppose $A$ is a  domain. If $A$ is not Bassian then $A$ has an infinite ascending sequence of
%If $A$ is finitely presented and not Hopfian, then $A$ has an infinite ascending sequence of completely prime
%finitely generated ideals.
%\end{lem}
%\begin{proof} In either case, one has a map $A \to A/J_1 \subsegteq A$ with $J_1 \ne 0.$ Since $A$ is a domain, $J$ is a completely prime ideal,
%and continuing the argument on $A/J_1$ (since it is isomorphic to $A$ yields the first assertion. The second assertion follows from \cite[Lemma]{ReSm}.
%\end{proof}

The ACC on semiprime ideals also comes up in \cite[Proposition
14]{Lo}. But we can push these ideas further.
 Following Malt'sev and Amitsur~\cite{Am}, we  define   property
${\operatorname{Emb}_n}$  to be that $A/J^i$ can be embedded in
matrix rings $M_n(K_i)$ over commutative ring $K_i$, with $n$ bounded.

 Amitsur~\cite[Theorem~1]{Am} proved
that any affine   commutative Noetherian algebra satisfies
ACC($\operatorname{Emb}_n$).
An algebra $A$ is called \textbf{weakly representable} if $A$
can be embedded into $M_n(K)$ for some commutative ring $K$.
 Consequently, one has:

 \begin{thm}\label{Hopf} Every   weakly representable affine
 algebra  over a commutative Noetherian ring  is Hopfian.
 \end{thm}
  \begin{proof} Apply Lemma \ref{makeswork} to Amitsur's theorem.
\end{proof}

%
%
% The following fact is well known, but we sketch the proof
%for the reader's convenience.
%
% \begin{lem}\label{Rep01}
% Any subdirectly irreducible commutative   affine algebra $C = C_0[ a_1, \dots, a_\ell]$
% over a Noetherian ring $C_0$
% is
%Artinian.
%\end{lem}
%\begin{proof} Let $I$ be the minimal ideal. If $I =0$ then $C$ is a
%field, so we may assume that $I \ne 0.$ Then $I^2 \ne I$ by
%Nakayama's lemma, so $I^2 = 0.$ As noted in \cite{RoSm}, Fitting's
%Lemma implies that all zero divisors are nilpotent.
%
%We claim that any regular element $c$ is invertible. Indeed, it is
%well known that there is $c' \in c C$ such that $(1-c') I = 0.$
%Hence, $(1-c')$ is nilpotent, implying $c' = 1 - (1-c')$ is
%invertible, implying $c$ is invertible.
%
%This means $C$ is its own classical ring of fractions, which is
%Artinian.
%\end{proof}
%%
%%The following is an  easy and well-known application of Noetherian
%%  induction, but we provide the proof for the reader's convenience.
%%
%% \begin{lem}\label{Rep01} Any ring $A$ satisfying the ACC on (two-sided) ideals
%%  is a finite subdirect product
%%  of irreducible rings each having a unique minimal nonzero ideal and prime rings.
%%\end{lem}
%%  \begin{proof} Otherwise take a counterexample $A$ such that the
%%  assertion holds for $A/A$ for each ideal $A$ of $A$. Then $A$ is
%%  not prime, by assumption, so there are ideals $A_1,A_2 \ne 0$ with
%%  $A_1A_2 = 0.$ But then taking $B_i/A_i$ minimal ideals in $A/A_i$,
%%  we see for any ideal $B$ if $A$ that $B_i \subseteq B+ A_i$,
%%  so $$(B+A_1)(B+A_2) \subseteq
%%  irreducible, since otherwise
%%\end{proof}
%

We can get information about the  kernel of an epimorphism from the
following observation.

\begin{lem}\label{J1}
$\varphi(J) \subseteq J$ for any epimorphism $\varphi$ of $A$, and thus
$\varphi$ induces an epimorphism $\overline{\varphi}$ of $A/J,$ under any of the following conditions:
\begin{itemize}
\item $J$ is the prime radical of $A$.

\item $J$ is the upper nilradical of $A$.

\item $J$ is the Jacobson radical of $A$.
 \end{itemize} \end{lem}
\begin{proof} $\varphi (J)$ is an
 ideal since $\varphi$ is onto,
and it is built up inductively from nilpotent  (resp.~nil,
resp.~quasi-invertible) ideals, so is contained in $J$. But this
means $\varphi$ induces a homomorphism  of $A/J$ to itself, which
clearly is onto.
\end{proof}

 \begin{prop}\label{GKcomp} Under the hypotheses and notation
  of Lemma~\ref{J1}, if  $\varphi$ is any epimorphism of $A$ with $A/J$ Hopfian, then $\ker \varphi \subseteq J$.
 \end{prop}
\begin{proof} $\ker  \varphi/J = 0.$
\end{proof}

\begin{cor} If $A$ satisfies ACC(semiprime ideals) and  $\varphi$ is an epimorphism of $A$, then $\ker \varphi$ is contained in the prime radical of $J$.
 \end{cor}
%
%\begin{rem} It is also of interest to examine the kernel of an epimorphism $\varphi$ of $A$ given that it is a certain kind of ideal, such as prime. Note that
%when $A$ is a domain, $\ker \varphi $ must be a completely prime ideal.  \end{rem}

\subsection{Hopfian modules}$ $

 Ironically, we often get more with a stronger
definition, by turning to modules.

\begin{defn} An   $A$-module $M$ is \textbf{Hopfian} if every epimorphism
 $M \to M$ is an isomorphism.
 \end{defn}

 Note that this generalizes the definition for rings, since we can
 always take $A$ to be an $A^{\operatorname{op}}\otimes A$ bimodule.

As above, for a given property designated P, we say  that a module
$M$ satisfies ACC(P) if  any   ascending chain  of submodules $N_1
\subseteq N_2 \subseteq \dots $ for which each $M/N_i$ has property
P, must stabilize. We have the following analog of Lemma~\ref{makeswork}.

\begin{lem}\label{makeswork1} For any given property P, if $M$ has property P
and every proper homomorphic image of $M$ satisfies ACC(P), then $M$
is Hopfian.\end{lem}
\begin{proof} Otherwise take $N_1$ such that $\varphi: M \cong M/N_1$, and
$N_2 = \varphi^{-1}(N_1) $ (so that $M \cong M/N_1 \cong  M/N_2$ and
so forth), but $N_1 \subset N_2 \subset \dots,$ a contradiction.
\end{proof}

\begin{cor} Any Noetherian module is
Hopfian. \end{cor}
\begin{proof} Property P is vacuous.  \end{proof}

In fact, any finitely generated module over a commutative ring is
Hopfian, by a theorem of Vasconcelos \cite{V}, with a short proof
given in \cite{OR}.
%
%But we can do better.
%
%
%\begin{cor} Any critical  module with Krull dimension is
%Hopfian. \end{cor}
%\begin{proof} Suppose we have $M/\ker f \conj M$ with $\ker f \ne 0.$
%Almost all of the factors $M/f(M),$ $f(M)/f^2(M),$ etc., must
%have smaller Krull dimension, but these are isomorphic to $M$.
%\end{proof}

\section{Counterexamples}

Many counterexamples are Hopf algebras, when they are built as group
algebras or enveloping algebras of Lie algebras. This provides a
method of showing that a group algebra or enveloping algebra does
not satisfy the ACC on ideals.

\subsection{Hopfian and non-Hopfian  groups}$ $

Analogously to our earlier definitions, we have:

\begin{defn} A group $G$ is \textbf{Hopfian} if every group epimorphism
 $G \to G$ is an isomorphism.
 \end{defn}

The free group is Hopfian, by \cite[Theorem~2.13]{MaKS}. A broader
class of examples is given in \cite[Corollary~2.13.2]{MaKS}; and
\cite[p. 111]{MaKS} provides other examples dating back to
B.H.~Neumann\cite{Ne}. An example of a finitely presented
non-Hopfian group $G$ is given in \cite[p. 260]{MaKS}: $G = \langle
b,t; t^{-1}b^2t = b^3\rangle$. (Indeed, defining $c =
b^{-1}t^{-1}bt$ one mods out by the relation $b = c^2$ to get the
image  $H = \langle c,t; t^{-1}c^2t = c^3\rangle$.)

  Ollivier and
Wise \cite{OW} embed any countable group into a non-Hopfian group
which is the outer automorphism group of a group satisfying
Kazhdan's property (T), and a subgroup of a hyperbolic group. De
Cornulier   produces an example of a finitely presentable,
non-Hopfian group with property (T) and infinite outer automorphism
group (but not a subgroup of a hyperbolic group), by means of the
following lemma:

\begin{lem}[{\cite[Lemma 2.3]{DeC}}]\label{nonH} Let $G$ be a group and $Z$ a central subgroup.
Any automorphism $\varphi$ of $G$   that   induces a surjective,
 endomorphism of $G/Z,$ whose kernel is $\varphi^{-1}(Z)/Z.$
\end{lem}
\begin{example}[{Cornulier \cite[Theorem~3.1]{DeC}}]\label{three} ${I_{m}} $ denotes the identity $m\times m$ matrix, each
$A_{jk} \in M_{j,k} (R)$ and $B_{jj} \in \SL_j(R).$
\begin{enumerate} \eroman
\item Let $G$ be the set of matrices over $R=\mathbb Z [\frac 1 p]$ of the form
$\( \begin{array}{cccc}  {I_{n_1}} & {A_{12}} & {A_{13}}
\\ 0 &  {B_{22}}  &  {A_{23}}  \\ 0   & 0   & {I_{n_3}}
\end{array} \)$.
The diagonal matrix $(p I_{n_1}, I_{n_2+n_3})$  acts by conjugation
on $G$, and sends   the center $Z =  \( \begin{array}{cccc}
{I_{n_1}}  & 0 & {A_{13}}
\\ 0 &  0  & 0   \\ 0    & 0
&  {I_{n_3}}
\end{array} \)$ to the proper subset $\( \begin{array}{cccc}
{I_{n_1}} & 0  & p{A_{13}}
\\ 0 &  0  &   0    \\ 0    & 0
&  {I_{n_3}}
\end{array} \),$ thereby yielding a
non-injective surjection (whose kernel is finite). Hence $G/Z$ is
non-Hopfian, and is seen in \cite[Proposition 2.7]{DeC} to be
finitely generated.

\item  For finite presentation De
Cornulier needs to add a component. Let $G$ be the set of matrices
over $R=\mathbb Z [\frac 1 p]$ of the form $\( \begin{array}{cccc}
{I_{n_1}} & {A_{12}} & {A_{13}} & {A_{14}}
\\ 0 &  {B_{22}}  & {A_{23}}&  {A_{24}}  \\ 0   & 0 &  {B_{33}}  &  {A_{A4}}  \\ 0   & 0  & 0
&  {I_{n_4}}
\end{array} \)$
The diagonal matrix $(p I_{n_1}, I_{n_2+n_3+n_4})$  acts by
conjugation on $G$, and sends   the center $Z$ to a proper subset of
$Z$, thereby yielding a non-injective surjection (whose kernel is
finite). Hence $G/Z$ is non-Hopfian, and is seen in
\cite[Proposition 2.7]{DeC} to be finitely generated. Finite
presentation requires an argument using dominant weights, given in
\cite[Lemma 3.2]{DeC}.
\end{enumerate}
\end{example}

\subsubsection{Algebra counterexamples arising from group theory}$ $

For any ring $R$, any group homomorphism $\varphi _N : G \to G/N$
extends naturally to a ring homomorphism $\widetilde \varphi : R[G]
\to R[G/N].$ Applying the reverse direction, we see that $\widetilde
\varphi$ is an isomorphism if and only if $\varphi _N$  is an
isomorphism.

\begin{lem} \label{nonHopfgp} If a group $G$ is non-Hopfian, then its group algebra $R[G]$ is non-Hopfian, for any ring $R$.\end{lem}
\begin{proof} For the group homomorphism $\varphi : G \to G/N$  with $N \ne \{ 1 \}$,
 the kernel of $\widetilde \varphi $ is easily seen to contain all
elements of the form $1-g $ for $g \in N$ (and in fact is generated
by them), so is nontrivial.
\end{proof}

Thus Example \ref{three} yields non-Hopfian  group algebras. Passman
and Small \cite[Theorem~2.3]{PS} complete a result of Baumslag,
showing that a group $G$ is finitely presented, iff the group
algebra $R[G]$ is finitely presented, so the finitely presented
(resp.~non-finitely presented) group counterexamples give finitely
presented (resp.~non-finitely presented) non-Hopfian group algebras.

On the other hand,  one way of obtaining Hopfian groups is to apply
the group version of Lemma~\ref{makeswork}.
Zelmanov~\cite[Theorems~2.3.2,\, 2.3.3\, 2.3.5]{Ze} shows that
various ascending chain conditions terminate, and thus the
corresponding kinds of groups are Hopfian.

\subsection{Hopfian and non-Hopfian Lie algebras}$ $

\begin{defn}
 A Lie algebra $L$ is \textbf{Hopfian} if every Lie epimorphism
 $L \to L
 $ is an isomorphism.
 \end{defn}
\subsubsection{Counterexamples arising from Lie theory}

\begin{lem}  \label{nonHopfLie} If a  Lie algebra $L$ is non-Hopfian, then its enveloping algebra $C[L]$ is non-Hopfian,
for any integral domain $C$, and consequently $C[L]$ does not
satisfy the ACC on prime ideals.\end{lem}
\begin{proof} As for Lemma~\ref{nonHopfgp}, since the kernel of a Lie homomorphism
extends to the kernel of the induced homomorphism of enveloping
algebras.
\end{proof}

For restricted Lie algebras, one can just take the  Lie version
(``i.e., logarithm'') of~Example~\ref{three} to get a non-Hopfian
restricted enveloping algebra.

\begin{example}\label{threeLie} ${I_{m}} $ denotes the identity $m\times m$ matrix, each
$A_{jk} \in \gl_{j,k} (R)$ and $B_{jj} \in \sli_j(R).$
\begin{enumerate} \eroman
\item Let $L$ be the set of matrices over $R=\mathbb Z [\frac 1 p]$ of the form
$\( \begin{array}{cccc}  0 & {A_{12}} & {A_{13}}
\\ 0 &  {B_{22}}  &  {A_{23}}  \\ 0   & 0   & 0
\end{array} \)$.
The diagonal matrix $(a^{[p]}, 0, 0, 0)$  acts by the adjoint action
on $L$, and sends   the center $Z =  \( \begin{array}{cccc} 0&  0 &
{A_{13}}
\\ 0 &  0  & 0  \\ 0   & 0 &  0
\end{array} \)$ to a proper subset, thereby yielding a
non-injective surjection (whose kernel is finite). Hence $L/Z$ is
non-Hopfian.

\item  For finite presentation  one could follow
Example\ref{three}(ii), but this would require implementing the
theory of Lie weights as in \cite[Theorem~6.4.3]{Ab2}, which we have
not yet verified.
\end{enumerate}
\end{example}

 Passman and Small
\cite[Lemma~3.2]{PS} give a semi-direct product
 modification $\mathcal L$  of Example~\ref{threeLie}(i),  for which
$\mathcal L/Z(\mathcal L)$ is non-Hopfian (not finitely presented)
for the analogous reason, taking the Lie version of
Lemma~\ref{nonH}. (The suitable adjoint map is an isomorphism with a
nontrivial kernel on the center.)

\subsubsection{Hopf counterexamples}$ $

Since group algebras and enveloping algebras are both examples of
Hopf algebras, one might ask if there is a common Hopf algebra
generalization, and indeed Passman and Small \cite[\S~5]{PS} present
such a candidate which they call $\mathcal R$.  Note that here
$Z(\mathcal R)$ is the component in the $1,3$ position, which is an
ideal.

\begin{conj}$ $ \begin{itemize}
\item The analog  to Example~\ref{three}(i) yields a semi-direct product $\mathcal R$ and an
  algebra $\mathcal R/Z(\mathcal R)$   that is non-Hopfian in the
sense of this paper.

\item  The analog $H$ of Example~\ref{three}(ii) yields a finitely
presented  Hopf algebra $H/Z(H)$ that is non-Hopfian in the sense of
this paper.
\end{itemize}
\end{conj}

  \section{Bassian rings}

Bass~\cite{Ba} strengthened the Hopfian property.

\begin{defn} A ring $A$ is \textbf{Bassian}
if for any injection $f:A \to A/J$ for $J \triangleleft A$, we must
have $J = 0.$
 \end{defn}

\begin{lem} Any Bassian ring is Hopfian.\end{lem}
\begin{proof} Given an epimorphism
$\varphi :A \to A$ which is not injective, we take $J = \ker
\varphi$ and have an isomorphism $A/J \to A,$ implying $J = 0$.
\end{proof}

In this way, Bassian is a natural generalization of Hopfian.

\begin{lem} Any Bassian ring satisfies the following property:
If  $\varphi :A \to A$ is a homomorphism   and $J \triangleleft A$
with $J \cap \ker \varphi = 0,$ then $J = 0.$
\end{lem}
\begin{proof} The composite of $\varphi$ with the canonical map $A
\to A/J$ is an injection, so $J= 0.$

\end{proof}
%We
%look for examples.

One gets instances of Bassian rings by means of growth.
%
%Recall that the \textbf{Gelfand-Kirillov dimension}  $\GKdim (A)$ of
%an affine algebra $A = F\{ a_1, \dots, a_\ell \}$ is
%\begin{equation} \label{GKdim}
%\GKdim (A): = \underset{n \to \infty}\varlimsup\log _n \tilde d_n,
%\end{equation}

\begin{lem}\label{easylem} Suppose $A$ is an affine algebra in which
the growth of $A/I$ is less than the growth of $A$, for each ideal
$I$ of $A$.  Then $A$ is Bassian.

In particular, if $\GKdim (A/I)< \GKdim(A)$ for all   ideals $I$ of
$A$, then $A$ is Bassian.
\end{lem}
\begin{proof} $A \cong \varphi
(A) \cong A/ \ker \varphi,$ but $\varphi (A)$ and $A$ have the same
growth rates, implying $ \ker \varphi = 0.$
\end{proof}

The   hypothesis of Lemma~\ref{easylem} holds for prime affine
PI-algebras over a field, cf.~\cite[Theorem~11.2.12]{BR}, so we
have:

\begin{corollary} Any prime affine PI-algebra over a field is Bassian.\end{corollary}

Likewise, we have:

\begin{corollary}\label{easy7} Any algebra with just infinite GK-dimension is Bassian. \end{corollary}
\begin{proof} Again, Lemma~\ref{easylem} applies.
\end{proof}

%\begin{corollary} Any prime affine  algebra over a field, having rational Hilbert series, is Bassian.\end{corollary}

\begin{corollary} Any prime affine  domain of finite GK-dimension is Bassian.\end{corollary}
\begin{proof} It is Ore, by a result of Jategaonkar \cite[Proposition~6.2.15]{row}, so any ideal $I$ contains a
regular element and $\GKdim (A/I)< \GKdim(A)$ by
\cite[Proposition~6.2.24]{row}.
\end{proof}

Recall from \cite{Mo} that an
 algebra   is called an \textbf{almost PI-algebra} if $ A$
is not PI (polynomial identity) but $A/I$ is PI for each nonzero
ideal $I$ of $A.$

Likewise, an  algebra over a field  is \textbf{nearly finite
dimensional} if $ A$ is not finite dimensional but $A/I$ is finite
dimensional for each nonzero  ideal $I$ of $A,$ \cite{FS}.

\begin{proposition}Any almost PI-algebra (resp.~nearly finite
dimensional algebra) is Bassian. \end{proposition}
\begin{proof} $A$ is not PI but $A/I$ is  PI, so cannot be isomorphic to $A$. Likewise for nearly finite
dimensional algebras.
\end{proof}

But here  are examples of a different flavor altogether.

\begin{lem} \label{nonHopfgp1} If a group $G$ is non-Bassian, then its group algebra $R[G]$ is non-Bassian, for any ring $R$.

If a Lie algebra is non-Bassian, then its enveloping algebra is
non-Bassian.\end{lem}
\begin{proof} The same as for Lemma~\ref{nonHopfgp}.
\end{proof}

\begin{example} \begin{enumerate} \eroman\item  Although Hopfian, the free algebra $\mathcal F_m : = F\{ x_1, \dots,
x_m\}$ ($m \le 3$) is far from Bassian. Indeed, it is well-known
that  $\mathcal F_2$ contains the free algebra as generated by
$\{\hat x_k  = x_1 x_2^k x_1 : k \ge 1 \}$, so for $J = \langle x_k:
k \ge 3 \rangle$, the map $\mathcal F_k \to \mathcal F_k / J$ given
by $x_k \mapsto \hat x_k$ is injective.
\item The Lie algebra $\mathcal W$ from the proof of \cite{PS} is non-Bassian.
\end{enumerate}
\end{example}

%
%On the other hand, \cite{AFS} provides an affine PI-algebra over a
%field, that  is not even Hopfian.

\begin{lem}\label{Noethind}
If $A$ satisfies the ACC on semiprime ideals, then there is an ideal
$J \triangleleft A$ maximal for which we have an embedding $f:
 A \to A/J$, and that $J$ is a semiprime ideal.
 \end{lem}
\begin{proof}
 Take a semiprime ideal $J_0 \triangleleft A$ maximal for  semiprime ideals for which we have an embedding $f:
 A \to A/J_0$. We claim that $J = J_0.$ Otherwise there is some $J_1 \triangleleft A$ properly containing $J_0,$ maximal for
which we have an embedding $f:
 A \to A/J_1$. The composite $\varphi: A \overset f
 {\to}
 A/{J_1} \to A/\sqrt {J_1}$ is injective. Indeed, $\sqrt {J_1}/{J_1}$, if not 0, contains a nonzero nilpotent ideal  of $A/{J_1}$,
 which
  thus intersects $f(A)$
 at a nilpotent ideal of $f(A) \cong A$ which is thus 0.
  But then $$ A \overset f \to A/{J_1}  \to A/\sqrt {J_1} $$
is injective,  implying that we could replace ${J_1}$ by $\sqrt
{J_1}$, which is semiprime, so by maximality of
  $J_0$, we see that $\sqrt J_1=J_0,$ a contradiction.
\end{proof}

 Bass~\cite[Theorem~1.1]{Ba} proved that if $A$ is
commutative reduced and finitely presented, then it is Bassian. The
main result of this section is a noncommutative generalization.

 \begin{thm}\label{Rep11}
 Any semiprime  affine
 PI-algebra $A$ over a field $K$ is Bassian.\end{thm}
 \begin{proof} We recall some well-known basic facts about a semiprime affine PI-algebra   $A$ over a field, taken from \cite{row0}.
 \begin{itemize} %(Procesi)
\item    $A$  satisfies the ACC on semiprime
 ideals, and thus, $A$ has a finite number of minimal prime ideals
 $P_1 \cap \dots \cap
  P_\ell,$
 whose intersection is 0. \item  $A$ is Goldie, and its
 ring of fractions $Q(A)$ is obtained by inverting regular central
 elements, and   $$Q(A) \cong  A_1 \times \dots \times A_\ell$$
where $A _i = Q(A)/M_i = Q(A/P_i)$ are the simple components, with
$M_i$ a maximal ideal of $Q(A)$.
\item  If $J$ is an annihilator left ideal of $A$, then $Q(A)J$ is an annihilator left ideal of
$Q(A)$, and taking right annihilators yields a 1:1 correspondence
between left and right annihilators.

\end{itemize}

 We follow Bass' proof, using the PI-structure theory
 where appropriate. Given an injection $A \to A/J$, we want to show that
 $J = 0.$
By Lemma~\ref{Noethind} we may assume that $J$ is semiprime, since
$A$ satisfies the ACC on semiprime ideals.
  Hence, for some $t,$ $$J = \sqrt J = P'_1 \cap \dots \cap
  P'_t,$$ an irredundant intersection of prime ideals of $A$, which thus
  are minimal prime over $J$. But
  furthermore $J$ cannot contain any regular elements since $\GKdim
  A/J \ge \GKdim A.$ Hence $\hat J : = J Z(Q(A))$ is a proper ideal of $Q(A)$ %(the
%  classical ring of central fractions of the semiprime Goldie PI-ring $A$)
 and thus has some nonzero left and right annihilator ideal $\hat J'$
in $Q(A)$, which is  precisely the sum of those simple components of
$Q(A)$ that miss $J$. Let $\widehat{J'} = \Ann _{Q(A)} \hat J = \Ann
_{Q(A)} J $. (This is the same if we take the annihilator from the
left or right.) Then $J' := A \cap\widehat{J'}$ is   $\Ann _A J$, so
is also a semiprime ideal and $J = \Ann _A J'$. Also note that $\hat
J$ annihilates $J'$, so $A \cap \hat J = J.$

Reordering the $P'_i$, we may assume that $\sqrt{\hat J} = \cap _{i
= 1}^j P_i',$ and $\sqrt{\widehat{J'}} = \cap _{i = j+1}^t P_i',$
for suitable $j.$ Then $\cap _{i = 1}^j P_i'$ is radical over $A
\cap \hat J = J.$

Let $P_i = A \cap P_i'$, $1 \le i \le t,$ which are prime ideals of
$A$.  Then $A \cap \hat J = \cap _{i = 1}^j P_i = J.$ It follows
that $A \hookrightarrow A/P_1 \, \times \dots \times A/P_t.$

We consider the function $\ell (R)$ for the length of a maximal
chain of left annihilator ideals of a ring $R$. It is well-known
that $\ell (W) \le \ell (R)$ for any subring $W$ of $R$, since any
chain $ \Ann _R (S_1) \subset \Ann _R (S_1) \subset  \dots$ lifts to
a  chain $ \Ann _W (S_1) \subset \Ann _W (S_1) \subset  \dots$
(which will intersect back to the original chain). In particular,
noting that $$A \overset f \hookrightarrow f(A) \subseteq A/J
\hookrightarrow Q(A/J) \cong A/P_1 \, \times \dots \times A/P_j ,$$
implying $\ell(A) \le \sum _{i=1}^j \ell (A/P_i).$

On the other hand, $\cap _{i = 1}^t P_i $ is nilpotent modulo $J\cap
J' = 0,$ implying  $\cap _{i = 1}^t P_i  = 0,$ so $$\ell(A) = \ell(
A/P_1 \, \times \dots \times A/P_t) = \sum _{i=1}^t \ell (A/P_i).$$
Together we get $\sum _{i=1}^t \ell (A/P_i) \le \sum _{i=1}^j \ell
(A/P_i),$ implying $j = t.$ Hence $J = 0,$ as desired.

%
%Let $P_1, \dots, P_m$ be the minimal prime ideals of $A$,  and $I =
%\{ i: J  \subseteq P_i\},$  and $I' = \{ i: J'   \subseteq P_i\}.$
%Then ????

\end{proof}

%\begin{rem}
Reviewing the proof, we can make do with the following hypotheses.
We leave the details to the reader.

 \begin{itemize}
\item    $A$  satisfies ACC(semiprime
 ideals).  \item  $A$ is Goldie.
\item There is some dimension function such that $\dim A \le \dim A'$ whenever $A
\subseteq A'$, and any semiprime ideal of $A$ lifts up to a proper
ideal of $Q(A)$ (i.e., does not contain regular elements).
\end{itemize}

The last condition is rather technical. In order to guarantee it,
Bass used transcendence degree, and we used GK dimension in
Theorem~\ref{Rep11}, but we are not familiar  with other dimensions
satisfying these properties, so we settle for the following result:

\begin{thm} Any semiprime  affine Goldie
 algebra  satisfying  ACC(semiprime
 ideals) and of finite GK dimension
 is Bassian.
\end{thm}
\begin{proof} (sketch) $A$ has a finite number of minimal prime ideals
 $P_1 \cap \dots \cap
  P_m,$
 whose intersection is 0.  The
 ring of fractions $Q(A)$ satisfies  $$Q(A) \cong  A_1 \times \dots \times A_m$$
where $A _i = Q(A)/M_i = Q(A/P_i)$ are the simple components, with
$M_i$ a maximal ideal of $Q(A)$.

Assume that $A$ is embedded in $A/J$. $J$ cannot contain any regular
elements since $\GKdim
  A/J \ge \GKdim A.$ Hence $J$ is contained in a proper ideal of
$Q(A)$ which, being semisimple, satisfies the ACC on annihilator
ideals. A Faith-Utumi type argument \cite{FU} shows that
$\Ann_{Q(A)}J$ is the sum of those components of $Q(A)$ disjoint
from $J$, and they intersect to $\Ann_{ A}J$.

Now we follow the previous proof.
\end{proof}

Unfortunately, the most familiar examples satisfying these
hypotheses fit into the previous known results: Affine semiprime
PI-algebras, and enveloping algebras of finite dimensional Lie
algebras (since they are Noetherian).

\section{Open questions}

The following questions are related to the Hopfian and Bassian properties.
\begin{itemize}
\item What condition on an algebra $A$ satisfying ACC(prime ideals) makes it Hopfian?
(Being affine PI is not enough, in view of Remark~\ref{Hopf rem}.) Is it enough that
$A$ is finitely presented?
\item Does the enveloping algebra of the centerless Virasoro algebra satisfy ACC(prime ideals)\footnote{Now proved
for completely prime ideals by Iyudu, N.K.~ and Sierra
\cite[Proposition~6.4]{IS}}? ACC(semiprime ideals)? ACC(ideals)?
\item Is the enveloping algebra of the centerless Virasoro algebra Bassian?
Hopfian? \footnote{Now proved along with the positive Witt algebra,
and also the Virasoro algebra being Hopfian, by Iyudu, N.K.~ and
Sierra \cite[Proposition~6.5]{IS}}
\end{itemize}

\end{document}